\documentclass[11pt,reqno]{amsart}
\usepackage{amssymb}
\usepackage{amsthm}
\usepackage{color}
\usepackage{eucal}
\usepackage{hyperref}

\newcommand{\R}{\mathbb R}

\newcommand{\ol}{\mathcal  L}
\newcommand{\C}{\mathbb C}

\newcommand{\F}{\mathcal{F}}
\newcommand{\p}{\partial}
\newcommand{\intt}{\int_0^T}

\newcommand{\str}{L^3_TL^{\infty}_{xy}}
\newcommand{\kse}{L^{\infty}_xL^2_{yT}}
\newcommand{\mne}{L^2_xL^{\infty}_{yT}}
\newcommand{\tres}{|\!|\!|}

\newcommand{\sech}{\textnormal{\,sech}}
\newcommand{\seccion}[1]{\section{#1}\setcounter{equation}{0}}
\numberwithin{equation}{section}

\newtheorem{theorem}{Theorem}[section]
\newtheorem{proposition}[theorem]{Proposition}
\newtheorem{remark}[theorem]{Remark}
\newtheorem{lemma}[theorem]{Lemma}

\begin{document}
\title[Surface Electromigration Equation] {Existence of solutions for the surface electromigration equation}
\author[F. Linares]{Felipe Linares}
\author[A. Pastor]{Ademir Pastor}
\author[M. Scialom]{Marcia Scialom}
\address{IMPA, Estrada Dona Castorina 110,  CEP 22460-320, Rio de Janeiro, RJ, Brazil}
\email{linares@impa.br}
\address{IMECC-UNICAMP, Rua Sérgio Buarque de Holanda 651, CEP 13083-859,
Campinas, SP, Brazil}
\email{apastor@ime.unicamp.br}
\address{IMECC-UNICAMP, Rua Sérgio Buarque de Holanda 651, CEP 13083-859,
Campinas, SP, Brazil}
%\date{11/03/05}
\email{scialom@ime.unicamp.br}

\begin{abstract}   We consider a model that describes electromigration in
nanoconductors known as surface electromigration (SEM) equation. Our
purpose here is to establish local well-posedness for the associated initial value
problem in Sobolev spaces from two different points of view. In the first one, we study the pure Cauchy problem
and  establish local well-posedness in $H^s(\R^2)$, $s>1/2$. In the second one, we study the Cauchy problem on the background 
of a Korteweg-de Vries solitary traveling wave in a less regular space. To obtain our results  we  make use of the smoothing properties of solutions for  the linear  problem corresponding to the Zakharov-Kuznetsov equation  for the latter problem. For the former problem we use bilinear estimates in Fourier restriction spaces introduced in \cite{mp}.
\end{abstract}



\maketitle

\seccion{Introduction}

In    this    note   we   consider   the   initial   value   problem   (IVP) for the surface
electromigration (SEM) equation,
\begin{equation}\label{zkp}
\begin{cases}
u_t+\p_x \Delta u+\dfrac{1}{2}\left(uu_x+u_x\phi_x+u_y\phi_y\right)=0,\qquad (x,y)\in\R^2,\;t\in\R,\\
\Delta\phi=u_x.
\end{cases}
\end{equation}
Here    $u=u(x,y,t)$    and    $\phi=\phi(x,y,t)$    are    real-valued   functions, $\Delta$ represents the two-dimensional Laplacian operator and subscripts stand for partial derivatives. The above system was derived by Bradley \cite{br1}, \cite{br2} to describe
electromigration  in nanoconductors and  it couples   the
Zakharov-Kuznetsov (ZK)  equation 
\begin{equation}\label{zk}
v_t+ \p_x \Delta v+vv_x=0, \qquad (x,y)\in\R^2,\;t\in\R,
\end{equation}
with  a  potential  equation.

The function $u$ represents the surface
displacement  and  $\phi$  is  the  electrostatic  potential  on the surface
conductor.  The physical situation is as follows: when an unidirectional electrical current
passes  through  a  piece  of solid metal, collisions between the conduction
electrons  and  the metal atoms at the surface lead to drift of these atoms.
This  is  known  as  surface electromigration (SEM) and it may cause a solid
metal surface to move and deform producing undesirable surface instabilities. The equation for surface electromigration differs from ZK equation through the coupling to
the electrical potential $\phi$. 

The ZK equation \eqref{zk} describes the propagation of nonlinear
ion-acoustic  waves  in a magnetized plasma. It was formally derived by Zakharov and Kuznetsov in \cite{zk0}. Recently, a rigorous derivation of the ZK was given in \cite{lls} as a long wave limit of the Euler-Poisson system.  It may also be viewed as a
two-dimensional  version  of the Korteweg-de Vries (KdV) equation,
\begin{equation}\label{kdveq}
u_t+u_{xxx}+uu_x=0,
\end{equation}
which accounts weak lateral dispersion given by the term $v_{xyy}$.  The  phenomena  of
the moving free surface of a metal film in response  to  the electrical
current flowing through the bulk of the film is reminiscent of the way the
flow in the bulk of a fluid affects the motion of its  surface.  This
analogy, however, does  not match since the boundary conditions  are  very
different in  the two problems. Bradley \cite{br2}, Schimschak  et  al
\cite{sh1} and Gungor et al \cite{gn1} considered the propagation  of
solitons  over a free surface of a current-carrying metal film. Recently, M.
C. Jorge et al in \cite{mc} investigated the evolution of lump
solutions  for  the ZK and SEM equations. They derived
approximate  equations including the important effect of the radiation shed
by  the  lumps  as they evolve  and studied  the  evolution  of  the lump
disturbances asymptotically and numerically. Since SEM may cause electrical
failure  of a current carrying metal line it is interesting to have a better
understanding of  the properties  of  the  equation  because it will be an
important factor limiting   the   reliability  of  integrated  circuits.

\begin{remark}
From the mathematical point of view, \eqref{zkp} and \eqref{zk} may be viewed as a two-dimensional generalization of \eqref{kdveq}. Indeed, it is clear if $v$ does not depend on  the transverse variable $y$, then \eqref{zk} reduces immediately to \eqref{kdveq}. Also, if $u$ and $\phi$ do not depend on $y$ in \eqref{zkp} and $u$ and $\phi_x$ have a suitable decay to zero as $x\to\pm\infty$ then it follows from the second equation in \eqref{zkp} that $\phi_x=u$. Substituting this in the first equation, one sees that it also reduces to \eqref{kdveq}. 
\end{remark}

The notion of local well-posedness thorough this paper includes the properties of existence, uniqueness, persistence and continuous 
dependence upon the initial data.

\medskip

The IVP associated with the ZK equation \eqref{zk} has been widely studied in recent years. Indeed, the first result in this direction is due to Faminskii \cite{f}; he proved that \eqref{zk} is locally-in-time well-posed in the usual $L^2$-based Sobolev spaces $H^s(\R^2)$, $s\geq1$. In \cite{LP}, Linares and Pastor established that the IVP is locally well-posed for $s>3/4$. The Sobolev index was pushed down independently by Grunrock and Herr (\cite{gh}) and Molinet and Pilot (\cite{mp}), where the authors showed the local well-posedness for $s>1/2$. Recently, 
Kinoshita in \cite{kinoshita} established a sharp local well-posedness for data in $H^s(\R^2)$, $s>-1/4$.  We shall also mention that the
study of IVP associated to generalizations of the ZK equation has gained  a lot of attention recently.  We refer the reader to \cite{bl},\cite{flp}, \cite{Gru1}, \cite{Gru2}, \cite{HK}, \cite{LP1}, \cite{LSaut}, \cite{Kinoshita2}, \cite{rv}, \cite{RV2} and references therein.

\medskip

In  this work we focus on the study of the IVP associated with \eqref{zkp} from two different points of view. In the first one, we study the pure IVP; thus we couple \eqref{zkp} with the initial condition 
\begin{equation*}%\label{initialp}
u(x,y,0)=u_0(x,y)
\end{equation*}
 and study the IVP  in the Sobolev spaces  $H^s(\R^2)$.   Our goal will be
to  establish  a  local well-posedness theory in $H^s(\R^2)$,
$s>  1/2$.  Since   \eqref{zkp}  couples  a  ZK  equation with a
potential $\phi$ we will exploit the properties of solutions of ZK
equation to reach our purpose. 

The second point of view is closed related with the study of transverse instability of one-dimensional solitons. To put the set up forward, note that \eqref{zkp} has solitary-wave solutions given by
\begin{equation}\label{KdV-soliton}
u(x,y,t)=\varphi(x-\omega t)=3\;\omega \sech^2\left(\frac{\sqrt \omega}{2}(x-\omega t)\right), \quad \omega>0,
\end{equation}
\begin{equation}\label{KdV-soliton1}
\phi(x,y,t)=\psi(x-\omega t)=6\;\sqrt{\omega} \tanh\left(\frac{\sqrt \omega}{2}(x-\omega t)\right), \quad \omega>0.
\end{equation}

\begin{remark}\label{rema1}
Note that $\varphi$ is exactly the solitary-wave solution  of the KdV equation  \eqref{kdveq} and $\psi_x=\varphi$.
\end{remark}

To investigate the transverse instability of the traveling waves \eqref{KdV-soliton} and \eqref{KdV-soliton1} under localized perturbations, one needs to study the evolution of $v:=u-\varphi$ and $w:=\phi-\psi$. Substituting this transformations into \eqref{zkp}, we obtain the system
\begin{equation}\label{sem1}
\begin{cases}
v_t+\partial_x \Delta v+\dfrac{1}{2}(vv_x+v_xw_x+v_yw_y)+\dfrac{1}{2}(\varphi_xv+\varphi v_x)+\dfrac{1}{2}(\psi_xv_x+\varphi_xw_x)  =  0,  \\
\Delta w=v_x.
\end{cases}
\end{equation}
Thus, we are lead to study \eqref{sem1} coupled with an initial condition $v(x,y,0)=v_0(x,y)$ in $H^s(\R^2)$. This step is then necessary to understand the dynamics of the SEM equation on the background of a nonlocalized solitary traveling wave. 
 It should be noted that in \cite{br1} the author established that the traveling waves \eqref{KdV-soliton} and \eqref{KdV-soliton1} are unstable under sinusoidal perturbation with wave-vector perpendicular to the direction of propagation.

Next  we describe the strategy to solve our problems. First we define
the  operator $\ol:=(\Delta)^{-1}\p_x$ and rewrite systems \eqref{zkp} and \eqref{sem1}
as  a  single  equation. Thus we will consider the following equivalent problems:
\begin{equation}\label{uzkp}
\begin{cases}
u_t+\p_x \Delta u+\dfrac{1}{2}\Big(uu_x+ u_x\p_x\ol(u)+u_y\p_y\ol(u)\Big)=0,\\
u(x,y,0)=u_0(x,y).
\end{cases}
\end{equation}
and
\begin{equation}\label{semIVP}
\begin{cases}
u_t+\partial_x \Delta u+\dfrac{1}{2}\Big(uu_x+u_x\partial_x\mathcal{L}(u)+u_y\partial_y\mathcal{L}(u)\Big)+\dfrac{1}{2}(\varphi_xu+2\varphi u_x)+\dfrac{1}{2}\varphi_x\partial_x\mathcal{L}(u)  =  0, \\
u(x,y,0)=u_0(x,y).
\end{cases}
\end{equation}
In order to  solve  the  IVP  \eqref{uzkp} (a similar approach will take place for \eqref{semIVP})  we will use its integral equivalent equation,                                                               i.e.,
\begin{equation}\label{ie}
\begin{split}
u(t)=U(t)u_0-\dfrac{1}{2}\int_0^t
U(t-t')\Big(u\p_xu+u_x\p_x\ol(u)+u_y\p_y\ol(u)\Big)(t')\,dt',
\end{split}
\end{equation}
where  $U(t)=\exp(t(\p_x^3+\p_x\p_y^2))$  is the unitary group associated to
the linear  problem
\begin{equation}\label{lzk}
\begin{cases}
u_t+\p_x \Delta u=0,\qquad (x,y)\in\R^2,\;t\in\R,\\
u(x,y,0)=u_0(x,y).
\end{cases}
\end{equation}
In particular, it is not difficult so see that $\widehat{{U(t)u_0}}(\xi,\mu)=e^{it(\xi^3+\xi\mu^2)}{\widehat{u}}_0(\xi,\mu)$, where the hat represents the Fourier transform.

\medskip

 Our  local   well-posedness   result for \eqref{uzkp} reads  as   follows (see Notation below).

\begin{theorem}\label{lwp}
Let $u_0\in H^s(\R^2)$, $s>1/2$. There exist $T=T(\|u_0\|_{H^s})$ and a unique solution
$u$ of  the IVP  \eqref{uzkp} satisfying
\begin{equation*}
u\in   C([0,T]:H^s(\R^2))\cap X^{s,\frac{1}{2}+}_T.
\end{equation*}
Moreover,  for  any  $T'\in(0,T)$,  the data-to-solution map,  $u_0\mapsto  u(t)$, defined from a
neighborhood of $u_0\in H^s(\R^2)$ into the class $C([0,T']:H^s(\R^2))\cap X^{s,\frac{1}{2}+}_{T'}$ is
smooth.
\end{theorem}

The  idea to prove Theorem \ref{lwp} is to use the dispersive properties of the equation and the contraction mapping principle to obtain
a  fixed point associated to the integral equation \eqref{ie}. The main tool
to  succeed  in  our  purpose  is a bilinear estimate in Bourgain spaces  introduced in \cite{mp}.

\medskip

\begin{remark}  As far as we know the only conserved quantities satisfied by
solutions   of \eqref{zkp} are
\begin{equation*}
I_1(u)=\int_{\R^2} u(x,y,t)\,dxdy\quad and \quad I_2(u)=\int_{\R^2} u^2(x,y,t)\,dxdy.
\end{equation*}
Thus  the  lack  of  a  conservation law in high-order Sobolev spaces does not allow us to
extend  the  solutions  obtained  in  Theorem  \ref{lwp}  globally  in time. Since the flow associated to the SEM equation is conserved in $L^2(\R^2)$ one expects to establish a local theory in this space. Moreover, a scaling argument suggests to obtain local well-posedness 
for  $s>-1$. We do not know whether the method developed by Kinoshita in \cite{kinoshita} may apply to lower the regularity
in our case.  As we mentioned before his result shows sharp local well-posedness for the IVP associated to the ZK equation  for 
initial data in $H^s(\R^2)$ for $s>-1/4$.  To obtain a $L^2$ local theory for the IVP \eqref{uzkp} is indeed a challenging open problem.
\end{remark}

Regarding the IVP \eqref{semIVP} our main theorem reads as follows.

\begin{theorem}\label{lwp1}
Let $\varphi(x-ct)$ be the soliton of the KdV equation in \eqref{KdV-soliton}. Let $u_0\in H^1(\R^2)$. There exist $T=T(\|u_0\|_{H^1})$ and a unique solution
$u$ of  the IVP  \eqref{semIVP} satisfying
\begin{equation}\label{p1p1}
u\in   C([0,T]:H^1(\R^2)),
\end{equation}
\begin{equation}
\|\p^2   u\|_{\kse}<\infty,
\end{equation}
\begin{equation}
\|\nabla  u\|_{\str}<\infty,
\end{equation}
and
\begin{equation}\label{p4p4}
\|u\|_{\mne}<\infty.
\end{equation}
Moreover,  for  any  $T'\in(0,T)$  the  map  $u_0\mapsto  u(t)$ defined in a
neighborhood of $u_0\in H^1(\R^2)$ into the class defined by \eqref{p1p1}-\eqref{p4p4} is
smooth.
\end{theorem}

\begin{remark}
It will be clear from the proof of Theorem \ref{lwp1} below that we do not use the particular explicit form of the traveling waves in \eqref{KdV-soliton} and \eqref{KdV-soliton1}. Instead, we only use the properties that $\varphi(x-\omega t)$ belongs to $L^2_xL^\infty_T$, and  $\varphi(x-\omega t)$ and $\varphi_x(x-\omega t)$ are uniformly bounded in $x,t$. In particular Theorem \ref{lwp1} holds if $\varphi$ is any traveling-wave solution of the KdV equation with these two properties. For instance, $\varphi$ may be any $N$-soliton of the KdV equation.
\end{remark}

\begin{remark}
 We point out that any $N$-soliton of the KdV equation is also a solution of both the ZK equation \eqref{zk} and the Kadomtsev-Petviashvili (KP) equation,
$$
(u_t+u_{xxx}+uu_x)_x\pm u_{yy}=0.
$$ 
Similar results as the one in Theorem \ref{lwp1} for these two equations have appeared in \cite{LPS}, \cite{MST} and \cite{MST1}. The local results were also extended to global ones due to the conservation of the energy at the $H^1(\R^2)$ level.
\end{remark}

\begin{remark} Using the symmetrization introduced by Gr\"unrock and Herr in \cite{gh} it may be possible to establish local well-posedness
in $H^s(\R^2)$, $s>3/4$, (see for instance \cite{LPSi}). We choose to follow the approach presented here for two reasons. First, the aforementioned result  is not sharp. Secondly, after applying the change of variables in \cite{gh} the nonlinear terms in the new equation increases and the estimating process turns out
to be tedious.
\end{remark}

\medskip

This paper is organized as follows. Below we introduce the basic notation and give the linear estimates we need. Section \ref{seclwp} is devoted to prove Theorem \ref{lwp}. The proof relies on a fixed point argument in the Bourgain spaces. We focus on proving the main bilinear estimate necessary to establish Theorem \ref{lwp}. In the last section, Section \ref{seclwp1}, we give the proof of  Theorem  \ref{lwp1}. The tools are based on the Strichartz estimate, Kato's smoothing effect, and the maximal function presented in Lemma \ref{lines}.\\

\section{Notation and Linear Estimates}
Here we will introduce the main notation used throughout the paper and give some preliminary linear estimates. Besides the standard notation in partial differential equations, we use $c$ to denote various constants that may vary line by line. If $A$ and $B$ are two positive constants, the notation $A\lesssim B$ means that there is $c>0$ such that $A\leq cB$. Also, $A\wedge B:=\min\{A,B\}$ and $A\vee B:=\max\{A,B\}$.
 Given any $r\in\R$ we write $r+$  for $r+\varepsilon$, where $\varepsilon>0$ is a sufficiently small number. The mixed space-time norm is defined as (if $1\leq p,q,r<\infty$)
$$
\|f\|_{L^p_xL^q_yL^r_T}= \left( \int_\R \left( \int_\R\left(\int_0^T|f(x,y,t)|^rdt\right)^{q/r}dy \right)^{p/q}dx \right)^{1/p}
$$
with standard modifications if either $p=\infty$, $q=\infty$, or $r=\infty$. Similar spaces appear if one interchanges the order of integration. If two indices are equal we put them into the same space; for instance, if $q=r$ we denote the norm $\|\cdot\|_{L^p_xL^q_yL^q_T}$ by $\|\cdot\|_{L^p_xL^q_{yT}}$. For short, we set
$$
\|\p^2 u\|_{\kse}:=\|\p^2_x u\|_{\kse}+\|\p^2_{xy} u\|_{\kse}+\|\p^2_yu\|_{\kse}
$$
and
$$
\|\nabla     u\|_{\str}:=\|\p_x    u\|_{\str}+\|\p_yu\|_{\str}.
$$

The space-time Fourier transform of $u=u(x,y,t)$ will be denoted by $\F u=\F u(\xi,\mu,\tau)$ or $\widehat{u}(\xi,\mu,\tau)$, whereas the Fourier transform in space will be denoted by $\F_{xy}u$. As usual, $\F^{-1}$ and $\F_{xy}^{-1}$ represent the respective inverse Fourier transforms.

Let $\eta\in C_0^\infty(\R)$ be such that $0\leq \eta\leq1$, $\eta\equiv1$ on the interval $[-5/4,5/4]$ and $\textrm{supp}(\eta)\subset[-8/5,8/5]$. Now we set $\zeta(\xi)=\eta(\xi)-\eta(2\xi)$ and for $k\in \mathbb{N}^*=\{1,2,\ldots\}$ we define
$$
\zeta_{2^k}(\xi,\mu):=\zeta\left(2^{-k}|(\xi,\eta)|\right)\quad \mbox{and} \quad \rho_{2^k}(\xi,\mu,\tau):=\zeta\left(2^{-k}\left(\tau-(\xi^3+\xi\mu^2)\right)\right).
$$
For convenience we set $\zeta_1(\xi,\mu)=\eta(|(\xi,\mu)|)$ and $\rho_1(\xi,\mu,\tau)=\eta(\tau-(\xi^3+\xi\eta^2))$. It is easy to see that
$$
\sum_N\zeta_N(\xi,\mu)=1.
$$
Here and throughout the paper, any summation over the variables $N,L,K,M$ are suppose to be dyadic with $N,K,L,M\geq1$. The Littlewood-Paley multipliers on frequencies and modulations are defined as
$$
P_Nu:=\F_{xy}^{-1}\left(\zeta_N\F_{xy}(u)\right) \quad \mbox{and}\quad Q_Lu:=\F^{-1}\left(\rho_L\F(u)\right).
$$

Given $s,b\in\R$,  the Bourgain spaces $X^{s,b}$ are defined as the completion of the Schwartz space $\mathcal{S}(\R^3)$ under the norm
$$
\|u\|_{X^{s,b}}:=\left( \int_{\R^3}\langle\tau-(\xi^3+\xi\mu^2)\rangle^{2b} \langle|(\xi,\mu)|\rangle^{2s}|\widehat{u}(\xi,\mu,\tau)|^2d\xi d\mu d\tau \right)^{\frac{1}{2}},
$$
where $\langle x \rangle=1+|x|$. For $T>0$, if $u$ is defined on $\R^2\times [0,T]$ we set
$$
\|u\|_{X^{s,b}_T}:=\inf\{\|\widetilde{u}\|_{X^{s,b}}; \; \widetilde{u}:\R^2\times\R\to\C, \ \widetilde{u}|_{\R^2\times[0,T]}=u    \}.
$$

Next we state useful linear estimates to prove Theorems \ref{lwp} and \ref{lwp1}. The first result concerns linear estimates in Bourgain spaces.

\begin{lemma}\label{linearbou}
Let $s\in\R$ and $b>1/2$. Then,
\begin{equation*}
	\|\eta(t)U(t)f\|_{X^{s,b}}\lesssim \|f\|_{H^s} \qquad   (\text{Homogeneous Estimate})
	\end{equation*}
	and
	\begin{equation*}
		\left\|\eta(t)\int_0^tU(t-t')g(t')dt'\right\|_{X^{s,\frac{1}{2}+\delta}}\lesssim \|g\|_{X^{s,-\frac{1}{2}+\delta}}, \qquad
		(\text{Non-homogeneous Estimate})
	\end{equation*}
where $0<\delta<1/2$. In addition, given $T>0$ and $-1/2<b'\leq b<1/2$ we have
\begin{equation*}
\|u\|_{X^{s,b'}_T}\lesssim T^{b-b'}\|u\|_{X^{s,b}_T}.
\end{equation*}
\end{lemma}
\begin{proof}
These estimates are classical by now. See, for instance, \cite{ginibre}.
\end{proof}

The second result provides the linear estimates which are sufficient to prove Theorem \ref{lwp1}. 

\begin{lemma} \label{lines}
	Let $u_0\in L^2(\R^2)$. Then
\begin{equation}\label{s1}
\|U(t)u_0\|_{\str}\lesssim    \|u_0\|_{L^2_{xy}}\qquad   (\text{Strichartz's Estimate\,})
\end{equation}
and
\begin{equation}\label{s2}
\|\nabla  U(t)u_0\|_{\kse}\lesssim \|u_0\|_{L^2_{xy}}.\qquad
(\text{Kato's   Smoothing})
\end{equation}
In addition, if $u_0\in H^s(\R^2)$, $s>3/4$, and $T>0$ is fixed then there exists a constant $c(s,T)$ depending only on $s$ and $T$ such that
\begin{equation}\label{s3}
\|U(t)u_0\|_{\mne}\le c(s,T)\,\|u_0\|_{H^s}.\qquad (\text{Maximal Function})
\end{equation}
\end{lemma}
\begin{proof}
These estimates were proved by
Faminskii  \cite{f}  inspired in those ones obtained for the KdV equation by
Kenig, Ponce and Vega \cite{kpv}.
\end{proof}

\section{Proof of Theorem \ref{lwp}}\label{seclwp}

This section is devoted to prove Theorem \ref{lwp}.
As  we observed in  the  introduction we will use the equivalent integral equation \eqref{ie} and the contraction mapping principal to obtain our result.  The next bilinear estimates are the main tools in our analysis regarding the proof of Theorem \ref{lwp}.

\begin{proposition} \label{bilinear}
	Let $s>1/2$. There exists a small $\delta>0$ such that
	\begin{equation}\label{estimatex}
	\|u_xv\|_{X^{s,-\frac{1}{2}+2\delta}}\lesssim \|u\|_{X^{s,\frac{1}{2}+\delta}}\|v\|_{X^{s,\frac{1}{2}+\delta}}
	\end{equation}
	and
	\begin{equation}\label{estimatey}
	\|u_yv\|_{X^{s,-\frac{1}{2}+2\delta}}\lesssim \|u\|_{X^{s,\frac{1}{2}+\delta}}\|v\|_{X^{s,\frac{1}{2}+\delta}}.
	\end{equation}
\end{proposition}

 Before proving Proposition \ref{bilinear} we remind some relevant estimates in our arguments.

\begin{lemma}\label{lemmaL4}
	The following estimate holds
	$$
	\|u\|_{L^4_{xyt}}\lesssim \|u\|_{X^{0,\frac{5}{6}+}}.
	$$
\end{lemma}
\begin{proof}
	See \cite[Corollary 3.2]{mp}.
\end{proof}

\begin{lemma}
	Let $N_1$, $N_2$, $L_1$, $L_2$ be dyadic numbers. Then,
	\begin{equation}\label{NN}
	\|(P_{N_1}Q_{L_1}u)(P_{N_2}Q_{L_2}v)\|_{L^2} \lesssim (N_1\wedge N_2)(L_1\wedge L_2)^{\frac{1}{2}}\|P_{N_1}Q_{L_1}u\|_{L^2}\|P_{N_2}Q_{L_2}v\|_{L^2}.
	\end{equation}
	If, in addition, $N_2\geq 4N_1$ or $N_1\geq4N_2$ then
	\begin{equation}\label{NN1}
	\begin{split}
	\|(P_{N_1}Q_{L_1}u)&(P_{N_2}Q_{L_2}v)\|_{L^2}\\
	&\lesssim \frac{(N_1\wedge N_2)^{\frac{1}{2}}}{N_1\vee N_2} (L_1\vee L_2)^\frac{1}{2}(L_1\wedge L_2)^{\frac{1}{2}}\|P_{N_1}Q_{L_1}u\|_{L^2}\|P_{N_2}Q_{L_2}v\|_{L^2}.
	\end{split}
	\end{equation}
\end{lemma}
\begin{proof}
	See \cite[Proposition 3.6]{mp}.
\end{proof}

\begin{proof}[Proof of Proposition \ref{bilinear}]
	We will prove  only \eqref{estimatex}. It will be clear from the proof itself that \eqref{estimatey} may be established in a similar fashion. The proof follows the same arguments as the ones  in \cite[Proposition 4.1]{mp}. So, we only give the main steps and follow closely the notation in \cite{mp}.   It suffices, by duality, to show the estimate
	$$
	I\lesssim \|u\|_{L^2_{xyt}}\|v\|_{L^2_{xyt}}\|w\|_{L^2_{xyt}},
	$$
	where
	$$
	I=\int_{\R^6}\Gamma_{\xi,\mu,\tau}^{\xi_1,\mu_1,\tau_1}\widehat{w}(\xi,\mu,\tau)\widehat{u}(\xi_1,\mu_1,\tau_1)\widehat{v}(\xi_2,\mu_2,\tau_2)d\nu,
	$$
	the functions $\widehat{w},\widehat{u},\widehat{v}$ are nonnegative, and
	$$
	\Gamma_{\xi,\mu,\tau}^{\xi_1,\mu_1,\tau_1}=|\xi_1|\langle |(\xi,\mu)|\rangle^s \langle \sigma\rangle^{-\frac{1}{2}+2\delta} \langle|(\xi_1,\mu_1)|\rangle^{-s}\langle \sigma_1\rangle^{-\frac{1}{2}-\delta} \langle|(\xi_2,\mu_2)|\rangle^{-s}\langle \sigma_2\rangle^{-\frac{1}{2}-\delta}
	$$
	$$
	d\nu=d\xi\,d\xi_1\,d\mu\,d\mu_1\,d\tau\,d\tau_1, \quad \xi_2=\xi-\xi_1, \quad \mu_2=\mu-\mu_1, \quad \tau_2=\tau-\tau_1,
	$$
	$$
	\sigma=\tau-\omega(\xi,\mu), \qquad \sigma_i=\tau_i-\omega(\xi_i,\mu_i), \qquad \omega(\xi,\mu)=\xi^3+\xi\mu^2.
	$$
	The main difference when compared our situation and that of  Proposition 4.1 in \cite{mp}  is that we have $|\xi_1|$ in the definition of $\Gamma_{\xi,\mu,\tau}^{\xi_1,\mu_1,\tau_1}$ instead of $|\xi|$.
	By using dyadic decomposition, one may rewrite $I$ as
	\begin{equation*}
	I=\sum_{N,N_1,N_2}I_{N,N_1,N_2},
	\end{equation*}
	with
	$$
	I_{N,N_1,N_2}=\int_{\R^6}\Gamma_{\xi,\mu,\tau}^{\xi_1,\mu_1,\tau_1}\widehat{P_Nw}(\xi,\mu,\tau)\widehat{P_{N_1}u}(\xi_1,\mu_1,\tau_1)\widehat{P_{N_2}v}(\xi_2,\mu_2,\tau_2)d\nu.
	$$
	Using that $(\xi,\mu)=(\xi_1,\mu_1)+(\xi_2,\mu_2)$ one decomposes the sum into five cases.
	\begin{itemize}
		\item[(i)] $Low\times Low\rightarrow Low$ interactions:
		$$
		I_{LL\rightarrow L}=\sum_{N\leq4,N_1\leq4,N_2\leq4}I_{N,N_1,N_2}.
		$$
		\item[(ii)] $Low\times High\rightarrow High$ interactions:
		$$
		I_{LH\rightarrow H}=\sum_{4\leq N_2,N_1\leq N_2/4,N\sim N_2}I_{N,N_1,N_2}.
		$$
		\item[(iii)] $High\times Low\rightarrow High$ interactions:
		$$
		I_{HL\rightarrow H}=\sum_{4\leq N_1,N_2\leq N_1/4,N\sim N_1}I_{N,N_1,N_2}.
		$$
		\item[(iv)] $High\times High\rightarrow Low$ interactions:
		$$
		I_{HH\rightarrow L}=\sum_{4\leq N_1,N\leq N_1/4,N_2\sim N_1}I_{N,N_1,N_2}.
		$$
		\item[(v)] $High\times High\rightarrow High$ interactions:
		$$
		I_{HH\rightarrow H}=\sum_{N\sim N_1\sim N_2}I_{N,N_1,N_2}.
		$$
	\end{itemize}
	With the above decomposition we see that
	$$
	I=	I_{LL\rightarrow L}+	I_{LH\rightarrow H}+	I_{HL\rightarrow H}+I_{HH\rightarrow L} + 	I_{HH\rightarrow H},
	$$
	and our task is then to show that each term on the right-hand side is bounded by $\|u\|_{L^2}\|v\|_{L^2}\|w\|_{L^2}$.\\
	
	\textit{Estimate for $I_{LL\rightarrow L}$}. Since all frequencies are bounded this is the simplest case. Indeed, by Parseval's identity, H\"older's inequality and Lemma \ref{lemmaL4},
	\[
	\begin{split}
	I_{N,N_1,N_2}&\lesssim \left\| \left(\frac{\widehat{P_{N_1}u}}{\langle \sigma_1\rangle^{\frac{1}{2}+\delta}} \right)^\vee \right\|_{L^4} \left\| \left(\frac{\widehat{P_{N_2}v}}{\langle \sigma_2\rangle^{\frac{1}{2}+\delta}} \right)^\vee \right\|_{L^4}\|P_Nw\|_{L^2}\\
	&\lesssim  \|P_{N_1}u\|_{L^2}\|P_{N_2}v\|_{L^2}\|P_Nw\|_{L^2},
	\end{split}
	\]
	giving 
	\begin{equation*}
	I_{LL\rightarrow L}\lesssim \|u\|_{L^2}\|v\|_{L^2}\|w\|_{L^2}.
	\end{equation*}\\

	\textit{Estimate for $I_{LH\rightarrow H}$.} By using dyadic decomposition on the modulation variables one writes
	$$
	I_{N,N_1,N_2}=\sum_{L,L_1,L_2} I_{N,N_1,N_2}^{L,L_1,L_2},
	$$
	where
	$$
	I_{N,N_1,N_2}^{L,L_1,L_2}=\int_{\R^6}\Gamma_{\xi,\mu,\tau}^{\xi_1,\mu_1,\tau_1}\widehat{P_NQ_Lw}(\xi,\mu,\tau)\widehat{P_{N_1}Q_{L_1}u}(\xi_1,\mu_1,\tau_1)\widehat{P_{N_2}Q_{L_2}v}(\xi_2,\mu_2,\tau_2)d\nu.
	$$
	By using Cauchy-Schwarz in $(\xi,\mu,\tau)$ and \eqref{NN1} we obtain
	\[
	\begin{split}
	I_{N,N_1,N_2}^{L,L_1,L_2}&\lesssim N_1L^{-\frac{1}{2}+2\delta}N^sL_1^{-\frac{1}{2}-\delta}N_1^{-s} L_2^{-\frac{1}{2}-\delta}N_2^{-s}\|(P_{N_1}Q_{L_1}u)(P_{N_2}Q_{L_2}v)\|_{L^2}\|P_NQ_Lw\|_{L^2}\\
	& \lesssim N_1L^{-\frac{1}{2}+2\delta}N^sL_1^{-\frac{1}{2}-\delta}N_1^{-s} L_2^{-\frac{1}{2}-\delta}N_2^{-s}\frac{N_1^{\frac{1}{2}}}{N_2}L_1^{\frac{1}{2}}L_2^{\frac{1}{2}}\\
	&\qquad \qquad\times \|P_{N_1}Q_{L_1}u\|_{L^2}\|P_{N_2}Q_{L_2}v\|_{L^2}\|P_NQ_Lw\|_{L^2}\\
	&\lesssim L^{-\frac{1}{2}+2\delta}L_1^{-\delta}L_2^{-\delta}N_1^{-(s-\frac{1}{2})}\|P_{N_1}Q_{L_1}u\|_{L^2}\|P_{N_2}Q_{L_2}v\|_{L^2}\|P_NQ_Lw\|_{L^2},
	\end{split}
	\]
	where in the last inequality we used that $N\sim N_2$ and $\frac{1}{N_2}\lesssim \frac{1}{N_1}$. Consequently,
	\[
	\begin{split}
	I_{LH\rightarrow H}&\lesssim \sum_{L,L_1,L_2}L^{-\frac{1}{2}+2\delta}L_1^{-\delta}L_2^{-\delta} \sum_{N_1\leq N_2/4,N\sim N_2}N_1^{-(s-\frac{1}{2})}\|P_{N_1}Q_{L_1}u\|_{L^2}\|P_{N_2}Q_{L_2}v\|_{L^2}\|P_NQ_Lw\|_{L^2}\\
	&\lesssim \|u\|_{L^2}\sum_{N\sim N_2}\|P_{N_2}v\|_{L^2}\|P_Nw\|_{L^2}\\
	&\lesssim \|u\|_{L^2}\|v\|_{L^2}\|w\|_{L^2}.
	\end{split}
	\]\\

	\textit{Estimate for $I_{HL\rightarrow H}$.} This case is similar to the last one, because now we have $N\sim N_1$ and $N_2\leq N_1$, so that
	\[
	\begin{split}
	I_{N,N_1,N_2}^{L,L_1,L_2}\lesssim L^{-\frac{1}{2}+2\delta}L_1^{-\delta}L_2^{-\delta}N_2^{-(s-\frac{1}{2})}\|P_{N_1}Q_{L_1}u\|_{L^2}\|P_{N_2}Q_{L_2}v\|_{L^2}\|P_NQ_Lw\|_{L^2}.
	\end{split}
	\]\\

	\textit{Estimate for $I_{HH\rightarrow L}$.} We set $\widetilde{f}(\xi,\mu,\tau)=f(-\xi,-\mu,-\tau)$. In addition, interpolation between \eqref{NN} and \eqref{NN1} (see \cite[Eq. (4.20)]{mp}) gives, for $\theta\in[0,1]$,
	\begin{equation*}
		\begin{split}
			\|\widetilde{(P_{N_1}Q_{L_1}u)}&(P_{N}Q_{L}w)\|_{L^2}\\
			&\lesssim \frac{(N_1\wedge N)^{\frac{1}{2}(1+\theta)}}{(N_1\vee N)^{1-\theta}} (L_1\vee L)^{\frac{1}{2}(1-\theta)}(L_1\wedge L)^{\frac{1}{2}}\|P_{N_1}Q_{L_1}u\|_{L^2}\|P_{N_2}Q_{L_2}v\|_{L^2}.
		\end{split}
	\end{equation*}
	As in the last two cases we obtain
	\[
	\begin{split}
	I_{N,N_1,N_2}^{L,L_1,L_2}&\lesssim L^{-\frac{1}{2}+2\delta}L_1^{-\frac{1}{2}-\delta}L_2^{-\frac{1}{2}-\delta} \frac{N_1N^s}{N_1^sN_2^s} \|\widetilde{(P_{N_1}Q_{L_1}u)}(P_{N}Q_{L}w)\|_{L^2}\|P_{N_2}Q_{L_2}v\|_{L^2}\\
	&\lesssim  L^{-\frac{1}{2}+2\delta}L_1^{-\frac{1}{2}-\delta}L_2^{-\frac{1}{2}-\delta} \frac{N_1N^s}{N_1^sN_2^s} \frac{N^{\frac{1}{2}(1+\theta)}}{N_1^{1-\theta}}  (L_1\vee L)^{\frac{1}{2}(1-\theta)}(L_1\wedge L)^{\frac{1}{2}}\\
	&\qquad\qquad\times \|P_{N_1}Q_{L_1}u\|_{L^2}\|P_{N}Q_{L}w\|_{L^2} \|P_{N_2}Q_{L_2}v\|_{L^2}.
	\end{split}
	\]
	Assuming without loss of generality that $L=L_1\vee L$ and using that $N\leq N_1$ and $N_2\sim N_1$ we deduce
	\[
	\begin{split}
	I_{N,N_1,N_2}^{L,L_1,L_2}\lesssim L^{2\delta-\frac{\theta}{2}}L_1^{-\delta}L_2^{-\frac{1}{2}-\delta}  N_1^{-(s-\frac{1}{2}-2\theta)}
	\|P_{N_1}Q_{L_1}u\|_{L^2}\|P_{N}Q_{L}w\|_{L^2} \|P_{N_2}Q_{L_2}v\|_{L^2}.
	\end{split}
	\]
	Now we choose $\theta\in(0,1)$ and $\delta>0$ small enough such that $0<2\theta<s-\frac{1}{2}$ and $\delta<\frac{\theta}{4}$. Consequently, using Cauchy-Schwarz in $N$ and $N_1$,
	\[
	\begin{split}
	I_{HH\rightarrow L}&\lesssim \sum_{N\leq N_1/4,N_2\sim N_1} N_1^{-(s-\frac{1}{2}-2\theta)}
	\|P_{N_1}Q_{L_1}u\|_{L^2}\|P_{N}Q_{L}w\|_{L^2} \|P_{N_2}Q_{L_2}v\|_{L^2}\\
	&\lesssim  \|u\|_{L^2}\|v\|_{L^2}\|w\|_{L^2}.	
	\end{split}\]\\

	\textit{Estimate for $I_{HH\rightarrow H}$.} Note in this case we have $N\sim N_1\sim N_2$. So it does not matter if in the definition of $\Gamma_{\xi,\mu,\tau}^{\xi_1,\mu_1,\tau_1}$ appears $|\xi_1|$ or $|\xi|$ the estimate is the same, because we can always replace $N_1$ by $N$. Thus this estimate is exactly the same one given in \cite[page 358]{mp}, that is,
	$$
	I_{HH\rightarrow H}\lesssim  \|u\|_{L^2}\|v\|_{L^2}\|w\|_{L^2}.
	$$
	
	Collecting all estimates above, the proof of the proposition is completed.
\end{proof}

With Proposition \ref{bilinear} in hand we are able to apply the fixed point theorem in a closed ball of $X_T^{s,\frac{1}{2}+\delta}$ in order to prove Theorem \ref{lwp}. The proof is quite standard so we omit the details. We only point out that after localization (in time) of the right-hand side of \eqref{ie} and using Lemma \ref{linearbou} we only need the estimate
\begin{equation}\label{curcialestimate}
\|u\p_xu+u_x\p_x\ol(u)+u_y\p_y\ol(u)\|_{X^{s,-\frac{1}{2}+2\delta}}\lesssim \|u\|^2_{X^{s,\frac{1}{2}+\delta}}.
\end{equation}
To see that \eqref{curcialestimate} holds we  note   that  
\begin{equation}\label{m1m2}
m_1(\xi,\eta)=\frac{\xi^2}{\xi^2+\eta^2} \quad  \mbox{and} \quad
m_2(\xi,\eta)=\frac{\xi\eta}{\xi^2+\eta^2}
\end{equation}
are Fourier multipliers in $X^{s,\frac{1}{2}+\delta}$, $s\geq0$.  In particular, the operators $\p_x\ol$ and $\p_y\ol$ are bounded in $X^{s,\frac{1}{2}+\delta}$, that is, 
\begin{equation*}
\|\p_x\ol(u)\|_{X^{s,\frac{1}{2}+\delta}} \lesssim \|u\|_{X^{s,\frac{1}{2}+\delta}}
\end{equation*}
and
\begin{equation*}
\|\p_y\ol(u)\|_{X^{s,\frac{1}{2}+\delta}}\lesssim \|u\|_{X^{s,\frac{1}{2}+\delta}}.
\end{equation*}
Consequently, after applying Proposition \ref{bilinear}, we get
\[
\begin{split}
\|u\p_xu+u_x\p_x\ol(u)+u_y\p_y\ol(u)\|_{X^{s,-\frac{1}{2}+2\delta}}&\lesssim  \|u\|^2_{X^{s,\frac{1}{2}+\delta}}+\|u\|_{X^{s,\frac{1}{2}+\delta}}\|\p_x\ol (u)\|_{X^{s,\frac{1}{2}+\delta}}\\
&\quad+\|\p_y\ol(u)\|_{X^{s,\frac{1}{2}+\delta}} \\
&\lesssim \|u\|_{X^{s,\frac{1}{2}+\delta}}^2.
\end{split}
\]

The proof of Theorem \ref{lwp} is thus completed.

\section{Proof of Theorem \ref{lwp1}} \label{seclwp1}

This section is devoted to prove Theorem \ref{lwp1}.  First we will establish all the nonlinear estimates to use in the proof. For simplicity, let us define
\begin{equation}\label{e1}
E_1(u)(t):=\int_0^t U(t-t')(uu_x)(t')\,dt',
\end{equation}
\begin{equation}\label{e2}
E_2(u)(t):=\int_0^t U(t-t')(u_x\p_x\ol(u))(t')\,dt',
\end{equation}
\begin{equation}\label{e3}
E_3(u)(t):=\int_0^t U(t-t')(u_y\p_y\ol(u))(t')\,dt'
\end{equation}
\begin{equation}\label{e4}
	E_4(u)(t):=\int_0^tU(t-t')(\varphi_xu+2\varphi u_x)dt',
\end{equation}
and
\begin{equation}\label{e5}
	E_5(u)(t):=\int_0^tU(t-t')(\varphi_x\partial_x\mathcal{L}(u))dt'.
\end{equation}

\medskip

Next, we estimate $E_i$, $i=1,\ldots,5$ in the norms needed to close the argument in the fixed point theorem. Before proceeding we remark that functions $m_1$ and $m_2$ in \eqref{m1m2} are also Fourier multipliers in $H^s(\R^2)$, $s\geq0$. Thus the operators $\p_x\ol$ and $\p_y\ol$ are bounded in $H^s(\R^2)$, $s\geq0$.

\begin{lemma}\label{l0}
Let $E_1$ be defined as in \eqref{e1}. Then,
\begin{equation}\label{l0e1}
\begin{split}
\|E_1(u)\|_{L^{\infty}_TH^1}
\lesssim T^{2/3}\|\p_x       u\|_{\str}\|u\|_{L^{\infty}_TH^1}+T^{1/2}\,\|\p^2u\|_{\kse}\|u\|_{\mne},
\end{split}
\end{equation}
\begin{equation}\label{l0e1a}
\begin{split}
\|\p^2   E_1(u)\|_{\kse}
\lesssim T^{2/3}\|\p_x u\|_{\str}\|u\|_{L^{\infty}_TH^1}
+T^{1/2}\,\|\p^2u\|_{\kse}\|u\|_{\mne},
\end{split}
\end{equation}
\begin{equation}\label{l0e2}
\begin{split}
\|\nabla  E_1(u)\|_{\str}\lesssim T^{2/3}\|\p_xu\|_{\str}
\|u\|_{L^{\infty}_TH^1}
+T^{1/2}\|u\|_{\mne} \|\p^2u\|_{L^{\infty}_TL^2_{xy}},
\end{split}
\end{equation}
and
\begin{equation}\label{l0e3}
\begin{split}
\|E_1(u)\|_{\mne}\lesssim c(1,T)\,\Big\{T^{2/3}\|\p_xu\|_{\str}\|u\|_{L^{\infty}_TH^1}   +T^{1/2}\|u\|_{\mne}   \|\p^2u\|_{L^{\infty}_TL^2_{xy}}\Big\}.
\end{split}
\end{equation}
\end{lemma}

\medskip

\begin{proof} First we show estimate \eqref{l0e1}. From \eqref{e1}, Minkowski's inequality,   group properties and Leibniz'  rule we  have
\begin{equation}\label{l0e11}
\begin{split}
\|E_1(u)(t)\|_{H^1}&\le  \intt \|uu_x\|_{L^2_{xy}}\,dt
+ \intt  \big(\|uu_{xx}\|_{L^2_{xy}}+\|u_xu_x\|_{L^2_{xy}}\big)\,dt\\
& \quad + \intt\big(\|uu_{xy}\|_{L^2_{xy}}+\|u_yu_x\|_{L^2_{xy}}\big)\,dt.
\end{split}
\end{equation}
The first term on the right hand side of \eqref{l0e11} can be estimate as follows:
\begin{equation}\label{l0e12}
\begin{split}
\intt \|uu_x\|_{L^2_{xy}}\,dt  \lesssim
 \intt \|u\|_{L^2_{xy}}\|u_x\|_{L^{\infty}_{xy}}\,dt\lesssim       T^{2/3}\|\p_x    u\|_{\str}    \|u\|_{L^{\infty}_TL^2_{xy}},
\end{split}
\end{equation}
where  we  have  used  H\"older's  inequality  in  space  and  then in time.

Next we estimate the second term on the right-hand side of \eqref{l0e11}. An argument similar to the one applied
in                            \eqref{l0e12}                           yields
\begin{equation}\label{l0e13}
\begin{split}
\intt \|u_xu_x\|_{L^2_{xy}}\,dt \lesssim
 \intt  \|u_x\|_{L^2_{xy}}\|u_x\|_{L^{\infty}_{xy}}\,dt
\lesssim T^{2/3}\|\p_x   u\|_{\str}   \|u_x\|_{L^{\infty}_TL^2_{xy}}.
\end{split}
\end{equation}
H\"older's  inequality in time allows  us to obtain
\begin{equation}\label{l0e14}
\intt  \|uu_{xx}\|_{L^2_{xy}}\,dt  \lesssim T^{1/2}\|uu_{xx}\|_{L^2_TL^2_{xy}}\lesssim T^{1/2}\|u\|_{\mne} \|u_{xx}\|_{\kse}.
\end{equation}
The estimate of the third term on the right-hand side of \eqref{l0e11} follows the same
argument  as  in  \eqref{l0e13}  and  \eqref{l0e14}.  Thus
\begin{equation}\label{l0e15}
\begin{split}
\intt   \big(\|u_xu_y\|_{L^2_{xy}}+ \|uu_{xy}\|_{L^2_{xy}}\big)\,dt'
\lesssim   T^{2/3}\|\p_x   u\|_{\str}   \|u_y\|_{L^{\infty}_TL^2_{xy}}+          T^{1/2}\|u\|_{\mne}         \|u_{xy}\|_{\kse}.
\end{split}
\end{equation}
Combining \eqref{l0e11} with inequalities \eqref{l0e12}--\eqref{l0e15} we obtain    \eqref{l0e1}.

To establish estimate \eqref{l0e1a} we first apply  Minkowski's inequality, the smoothing effect \eqref{s2} and the argument above to get
\begin{equation*}
\begin{split}
\|\p^2E_1(u)\|_{\kse}&\le \intt\big(2\|\p_x(uu_x)\|+\|\p_y(uu_x)\|\big)\,dt'\\
&\lesssim  T^{2/3}\|\p_x  u\|_{\str}  \big(\|u_x\|_{L^{\infty}_TL^2_{xy}}+
\|u_y\|_{L^{\infty}_TL^2_{xy}}\big)\\
&\quad + T^{1/2}\|u\|_{\mne} \big(\|u_{xx}\|_{\kse}+\|u_{xy}\|_{\kse}\big).
\end{split}
\end{equation*}

The inequality \eqref{l0e2} follows using Minkowski's inequality, the smoothing effect
\eqref{s1}    and    the    argument    used    to   obtain   \eqref{l0e1}:
\begin{equation*}
\begin{split}
\|\nabla E_1(u)\|_{\str}&\le \intt\big(\|\p_x(uu_x)\|+\|\p_y(uu_x)\|\big)\,dt'\\
&\lesssim  T^{2/3}\|\p_x  u\|_{\str}  \big(\|u_x\|_{L^{\infty}_TL^2_{xy}}+
\|u_y\|_{L^{\infty}_TL^2_{xy}}\big)\\
&\quad + T^{1/2}\|u\|_{\mne} \big(\|u_{xx}\|_{\kse}+\|u_{xy}\|_{\kse}\big).
\end{split}
\end{equation*}

Finally, to get inequality \eqref{l0e3} we use group properties, the maximal function estimate \eqref{s3}    and  the argument above,
\begin{equation*}
\begin{split}
\|E_1(u)\|_{\mne}&=\|U(t)\Big(\intt      U(-t')(uu_x)(t')\,dt'\Big)\|_{\mne}\\
&\le c(1,T)\|\intt U(-t')(uu_x)(t')\,dt'\|_{H^1}\\
&\le c(1,T) \intt\|uu_x\|_{H^1}\,dt'\\
&\le  c(1,T)  T^{2/3}\|\p_x u\|_{\str} \big(\|u_x\|_{L^{\infty}_TL^2_{xy}}+
\|u_y\|_{L^{\infty}_TL^2_{xy}}\big)\\
&\quad +c(1,T) T^{1/2}\|u\|_{\mne} \big(\|u_{xx}\|_{\kse}+\|u_{xy}\|_{\kse}\big).
\end{split}
\end{equation*}
The proof of the lemma is thus completed.
\end{proof}

\begin{lemma}\label{l1}
Let $E_2$ be defined as in \eqref{e2}. Then,
\begin{equation}\label{l1e1}
\begin{split}
\|E_2(u)\|_{L^{\infty}_TH^1}
\lesssim T^{2/3}\|\p_x  u\|_{\str}\|u\|_{L^{\infty}_TH^1}+T^{1/2}\,\|\p^2u\|_{\kse}\|\p_x\ol(u)\|_{\mne},
\end{split}
\end{equation}
\begin{equation}\label{l1e1a}
\begin{split}
\|\p^2  E_2(u)\|_{\kse}
\lesssim T^{2/3}\|\p_x       u\|_{\str}\|u\|_{L^{\infty}_TH^1}
+T^{1/2}\,\|\p^2u\|_{\kse}\|\p_x\ol(u)\|_{\mne},
\end{split}
\end{equation}
\begin{equation}\label{l1e2}
\begin{split}
\|\nabla E_2(u)\|_{\str}\lesssim T^{1/2}\|\p^2u\|_{\kse}\|\p_x\ol(u)\|_{\mne}
 +T^{2/3}\|\p_xu\|_{\str} \|\p_xu\|_{L^{\infty}_TL^2_{xy}},
\end{split}
\end{equation}
and
\begin{equation}\label{l1e3}
\begin{split}
\|E_2(u)\|_{\mne}
\lesssim   c(1,T)\,\Big\{T^{2/3}\|\p_x     u\|_{\str}\|u\|_{L^{\infty}_TH^1}
+  T^{1/2}\,\|\p^2u\|_{\kse}\|\p_x\ol(u)\|_{\mne}\Big\}.
\end{split}
\end{equation}
\end{lemma}
\begin{proof}  The inequalities \eqref{l1e1}-\eqref{l1e3} are obtained using
the arguments applied to show the corresponding estimates in Lemma \ref{l0},
so   we  will  omit the details.
\end{proof}

\begin{lemma}\label{l2}
Let $E_3$ be defined as in \eqref{e3}. Then,
\begin{equation}\label{l2e1}
\begin{split}
\|E_3(u)\|_{L^{\infty}_TH^1}
\lesssim T^{2/3}\|\p_y       u\|_{\str}\|u\|_{L^{\infty}_TH^1}+cT^{1/2}\,\|\p^2u\|_{\kse}\|\p_y\ol(u)\|_{\mne},
\end{split}
\end{equation}
\begin{equation}\label{l2e1a}
\begin{split}
\|\p^2 E_3(u)\|_{\kse}\lesssim T^{2/3}\|\p_y  u\|_{\str}\|u\|_{L^{\infty}_TH^1}+T^{1/2}\,\|\p^2u\|_{\kse}\|\p_y\ol(u)\|_{\mne},
\end{split}
\end{equation}
\begin{equation}\label{l2e2}
\begin{split}
\|\nabla E_3(u)\|_{\str}\lesssim T^{1/2}\|\p^2u\|_{\kse}
\|\p_y\ol(u)\|_{\mne} +T^{2/3}\|\p_yu\|_{\str}    \|\p_y   u\|_{L^{\infty}_TL^2_{xy}},
\end{split}
\end{equation}
and
\begin{equation}\label{l2e3}
\begin{split}
\|E_3(u)\|_{\mne}\lesssim  c(1,T)\,\Big\{T^{2/3}\|\p_y  u\|_{\str}\|u\|_{L^{\infty}_TH^1} +       T^{1/2}\,\|\p^2u\|_{\kse}\|\p_y\ol(u)\|_{\mne}\Big\}.
\end{split}
\end{equation}
\end{lemma}
\begin{proof} To obtain estimates \eqref{l2e1}-\eqref{l2e3} we apply similar
arguments to the ones used to prove the corresponding estimates in Lemma \ref{l0}.
Thus we  will also omit the  details.
\end{proof}

\begin{lemma}\label{l3}
Let $E_1$, $E_2$ and $E_3$ be defined as above. Then,
\begin{equation}\label{ol1}
\begin{split}
\|\p_x \ol (E_1)\|_{\mne}+\|\p_y\ol(E_1)\|_{\mne}&\lesssim
c(1,T)\Big\{\,T^{2/3}\|\p_x        u\|_{\str}\|u\|_{L^{\infty}_TH^1}\\
&\quad+T^{1/2}\,\|\p^2u\|_{\kse}\|u\|_{\mne}\Big\},
\end{split}
\end{equation}
\begin{equation}\label{ol2}
\begin{split}
\|\p_x              \ol              (E_2)\|_{\mne}+\|\p_y\ol(E_2)\|_{\mne}&
\lesssim     c(1,T)\,\Big\{\,T^{2/3}\|\p_x    u\|_{\str}\|u\|_{L^{\infty}_TH^1}\\
&\quad+T^{1/2}\,\|\p^2u\|_{\kse}\|\p_x\ol(u)\|_{\mne}\Big\},
\end{split}
\end{equation}
and
\begin{equation}\label{ol3}
\begin{split}
\|\p_x\ol(E_3)\|_{\mne}+\|\p_y                             \ol(E_3)\|_{\mne}
&\lesssim    c(1,T)\,\Big\{T^{2/3}\|\p_y    u\|_{\str}\|u\|_{L^{\infty}_TH^1}\\
&\quad+T^{1/2}\,\|\p^2u\|_{\kse}\|\p_y\ol(u)\|_{\mne}\Big\}.
\end{split}
\end{equation}
\end{lemma}

\begin{proof} We will only prove \eqref{ol1} since the other inequalities follow similarly. We first observe
that  $\p_x  \ol$  and  $\p_y  \ol$  commute  with the operator $U(t)$. Then
using group properties,  \eqref{s3}, \eqref{A},
and  Minkowski's inequality we  get
\begin{equation*}
\begin{split}
\|\p_x \ol(E_1)\|_{\mne}&\le \|U(t)\Big(\p_x \ol\intt U(-t')(uu_x)(t')\,dt'\Big)\|_{\mne}\\
&\le  c(1,T)\,\|\p_x \ol\intt U(-t')(uu_x)(t')\,dt'\|_{H^1}\\
& \le c(1,T)\,\intt \|uu_x\|_{H^1}\,dt'.
\end{split}
\end{equation*}
The estimate \eqref{ol1} now follows by using the argument employed to obtain \eqref{l0e1}.
This  completes the  proof.
\end{proof}

We have completed all the estimates for the terms $E_1$, $E_2$ and $E_3$ essential for the proof of Theorem \ref{lwp1}. Note that terms $E_4$ and $E_5$ contain the functions $\varphi$ and $\varphi_{x}$.  Here a careful estimate need to be performed because the function $\varphi$ does not belong to $L^2(\R^2)$. The crucial property is that $\varphi(x-\omega t)$ belongs to $L^2_xL^\infty_T$, and $\varphi(x-\omega t)$ and $\varphi_x(x-\omega t)$ belong to $L^\infty_{xT}$.

\begin{lemma} \label{lemmaE4}
	Let $E_4$ be defined as in \eqref{e4}. Then,
	\begin{equation}\label{ll1}
	\|E_4(u)\|_{L^\infty_TH^1}\lesssim T\|u\|_{L^\infty_TH^1}+T^{1/2}\|\partial^2 u\|_{L^\infty_xL^2_{yT}},
	\end{equation}
	\begin{equation}\label{ll2}
	\|\nabla E_4(u)\|_{L^3_TL^\infty_{xy}}\lesssim T\|u\|_{L^\infty_TH^1}+T^{1/2}\|\partial^2 u\|_{L^\infty_xL^2_{yT}},
	\end{equation}
	\begin{equation}\label{ll3}
	\|\partial^2E_4(u)\|_{L^\infty_xL^2_{yT}}\lesssim T\|u\|_{L^\infty_TH^1}+T^{1/2}\|\partial^2 u\|_{L^\infty_xL^2_{yT}},
	\end{equation}
	\begin{equation}\label{l4}
	\|E_4(u)\|_{L^2_xL^\infty_{yT}}\lesssim c(1,T)\left\{T\|u\|_{L^\infty_TH^1}+T^{1/2}\|\partial^2 u\|_{L^\infty_xL^2_{yT}}\right\},
	\end{equation}
	and
	\begin{equation}\label{l5}
	\|\partial_x\mathcal{L}(E_4(u))\|_{L^2_xL^\infty_{yT}}+\|\partial_y\mathcal{L}(E_4(u))\|_{L^2_xL^\infty_{yT}}\lesssim  c(1,T)\left\{T\|u\|_{L^\infty_TH^1}+T^{1/2}\|\partial^2 u\|_{L^\infty_xL^2_{yT}}\right\}.
	\end{equation}
\end{lemma}
\begin{proof}
	For \eqref{ll1}, we have
	\begin{equation}  \label{2.4}
	\|E_4(u)(t) \|_{L^2_{xy}}  \lesssim  
		\int_0^T \|\varphi_xu\|_{L^2_{xy}} dt'+\int_0^T \|\varphi u_x\|_{L^2_{xy}} dt'   
	 =:   {\displaystyle I_1+I_2. }
	\end{equation}
	Now, from H\"older's inequality,
	\begin{equation}\label{I_2}  
	\begin{split}
	I_1 & = {\displaystyle  \int_0^T \|
		\|\varphi_x u\|_{L^2_{x}}\|_{L^2_{y}} dt' \leq \int_0^T \|
		\|\varphi_x\|_{L^\infty_x} \|u\|_{L^2_{x}}\|_{L^2_{y}} dt'  }
	\\
	& \leq   {\displaystyle  \int_0^T
		\|\varphi_x\|_{L^\infty_x} \|u\|_{L^2_{xy}} dt' \leq
		\|\varphi_x\|_{L^\infty_{xT}} \|u\|_{L^\infty_TH^1_{xy}}\,T \lesssim
		T \|u\|_{L^\infty_TH^1_{xy}},  }
	\end{split}
	\end{equation}
	and similarly,
	\begin{equation}\label{I_3}
	I_2 \leq \|\varphi\|_{L^\infty_{xT}} \|u\|_{L^\infty_TH^1_{xy}}\,T \lesssim
	 T \|u\|_{L^\infty_TH^1_{xy}}.
	\end{equation}
Note that the implicit constant in the estimates for $I_1$ and $I_2$ is independent of $T$ (because $\varphi(x-ct)$
	and $\varphi_x(x-ct)$ are uniformly bounded in $x,t$). Thus, from
	\eqref{2.4}--\eqref{I_3},
	\begin{equation}\label{PhiL2}
	\|E_4(u)(t) \|_{L^2_{xy}} \lesssim T \|u\|_{L^\infty_TH^1_{xy}}.
	\end{equation}
Also,
	\begin{equation*}
		\begin{split}
			\|\partial_xE_4(u)(t) \|_{L^2_{xy}} & \lesssim  {\displaystyle
				\int_0^T \|\partial_x(\varphi_xu+\varphi u_x)\|_{L^2_{xy}} dt'}\\
			&\lesssim \int_0^T \|\varphi_{xx}u+2\varphi_x u_x\|_{L^2_{xy}}+ \int_0^T\|\varphi
			u_{xx}\|_{L^2_{xy}}=:J.
		\end{split}
	\end{equation*}
	The first integral above can be estimated as in \eqref{I_2}. Thus,
	\begin{equation}   \label{J_2}
	\begin{split}
	J & \lesssim  {\displaystyle  T \|u\|_{L^\infty_TH^1_{xy}}+ T^{1/2}
		\|\varphi u_{xx}\|_{L^2_{xyT}} }  \\
	& \lesssim   {\displaystyle  T \|u\|_{L^\infty_TH^1_{xy}}+ T^{1/2} \|\|\|\varphi\|_{L^\infty_T}
		\|u_{xx}\|_{L^2_T}\|_{L^2_y}\|_{L^2_x} }  \\
	& \lesssim   {\displaystyle  T \|u\|_{L^\infty_TH^1_{xy}}+ T^{1/2}
		\|\|\varphi\|_{L^\infty_T}
		\|u_{xx}\|_{L^2_{yT}}\|_{L^2_x} }     \\
	& \lesssim   {\displaystyle  T \|u\|_{L^\infty_TH^1_{xy}}+ T^{1/2}
		\|\varphi\|_{L^2_xL^\infty_T}
		\|u_{xx}\|_{L^\infty_x L^2_{yT}} }   \\	
	& \lesssim   {\displaystyle  T \|u\|_{L^\infty_TH^1_{xy}}+ T^{1/2}
		\|\partial^2u\|_{L^\infty_x L^2_{yT}}. }
	\end{split}
	\end{equation}
	In the last inequality we have used that $\varphi(x-ct)\in L^2_xL^\infty_T$.
		Hence, from  \eqref{J_2}, we obtain
	\begin{equation}\label{ptxPhiL2}
	\|\partial_xE_4(u)(t) \|_{L^2_{xy}} \lesssim T\|u\|_{L^\infty_TH^1}+T^{1/2}\|\partial^2 u\|_{L^\infty_xL^2_{yT}}.
	\end{equation}
	Similarly, we get 
	\begin{equation}\label{ptyPhiL2}
	\|\partial_yE_4(u)(t) \|_{L^2_{xy}} \lesssim T\|u\|_{L^\infty_TH^1}+T^{1/2}\|\partial^2 u\|_{L^\infty_xL^2_{yT}}.
	\end{equation}
	Therefore, from \eqref{PhiL2}, \eqref{ptxPhiL2} and
	\eqref{ptyPhiL2}, we deduce the bound \eqref{ll1}.

	For \eqref{ll2}, we have
	\begin{equation}  \label{118}
	\begin{split}
	\|\nabla E_4(u)\|_{L^3_TL^\infty_{xy}} & \lesssim {\displaystyle \int_0^T\|\partial_x(\varphi_xu+\varphi u_x)\|_{L^2_{xy}}+ \int_0^T\|\partial_y(\varphi_xu+\varphi u_x)\|_{L^2_{xy}}}.
	\end{split}
	\end{equation}
	The first integral is estimated as that for $J$. Hence,
	\begin{equation*}
		\begin{split}
			\|\nabla E_4(u)\|_{L^3_TL^\infty_{xy}} & \lesssim  {\displaystyle T\|u\|_{L^\infty_TH^1}+T^{1/2}\|\partial^2 u\|_{L^\infty_xL^2_{yT}}+ \int_0^T\|\varphi_xu_y+\varphi u_{xy}\|_{L^2_{xy}}}\\
			& \lesssim  T\|u\|_{L^\infty_TH^1}+T^{1/2}\|\partial^2 u\|_{L^\infty_xL^2_{yT}}\\
			& \qquad+ T\|\varphi_x\|_{L^\infty_{xT}}\|u_y\|_{L^\infty_TL^2_{xy}}+T^{1/2}\|\varphi\|_{L^2_xL^\infty_T}\|u_{xy}\|_{L^\infty_xL^2_{yT}}\\
			&\lesssim  T\|u\|_{L^\infty_TH^1}+T^{1/2}\|\partial^2 u\|_{L^\infty_xL^2_{yT}}.
		\end{split}
	\end{equation*}
	As in \eqref{118}, we estimate
	\begin{equation*}
		\begin{split}
			\|\partial^2E_4(u)\|_{L^\infty_xL^2_{yT}}& \lesssim   {\displaystyle \int_0^T\|\partial_x(\varphi_xu+\varphi u_x)\|_{L^2_{xy}}+ \int_0^T\|\partial_y(\varphi_xu+\varphi u_x)\|_{L^2_{xy}}} \\
			&\lesssim T\|u\|_{L^\infty_TH^1}+T^{1/2}\|\partial^2 u\|_{L^\infty_xL^2_{yT}}.
		\end{split}
	\end{equation*}
	This proves \eqref{ll3}. Now, for \eqref{l4}, observe that
	\begin{equation*}
		\begin{split}
			\|E_4(u)\|_{L^2_xL^\infty_{yT}}& \leq  c(1,T)\int_0^T\|\varphi_xu+2\varphi u_x\|_{H^1}dt'.
		\end{split}
	\end{equation*}
	So, \eqref{l4} follows by a similar analysis as before.
	
	Finally, since $\partial_x\mathcal{L}$ and $\partial_y\mathcal{L}$ are bounded operators in $H^1(\mathbb{R}^2)$,
	\begin{equation*}
		\begin{split}
			\|\partial_x\mathcal{L}(E_4(u))\|_{L^2_xL^\infty_{yT}}&+\|\partial_y\mathcal{L}(E_4(u))\|_{L^2_xL^\infty_{yT}}\\
			&\leq c(1,T)\|\partial_x\mathcal{L}\int_0^tU(-t')(\varphi_xu+2\varphi u_x)dt'\|_{H^1} \\ &\qquad+c(1,T)\|\partial_y\mathcal{L}\int_0^tU(-t')(\varphi_xu+2\varphi u_x)dt'\|_{H^1}\\
			&\leq c(1,T)\int_0^T\|\varphi_xu+2\varphi u_x\|_{H^1}dt'.
		\end{split}
	\end{equation*}
	Thus, \eqref{l5} also follows from arguments already used. This completes the proof of the lemma.
\end{proof}

\begin{lemma} \label{lemmaE5}
	Let $E_5$ be defined as in \eqref{e5}. Then,
	\begin{equation}\label{l6}
	\|E_5(u)\|_{L^\infty_TH^1}\lesssim T\|u\|_{L^\infty_TH^1},
	\end{equation}
	\begin{equation}\label{l7}
	\|\nabla E_5(u)\|_{L^3_TL^\infty_{xy}}\lesssim T\|u\|_{L^\infty_TH^1},
	\end{equation}
	\begin{equation}\label{l8}
	\|\partial^2E_5(u)\|_{L^\infty_xL^2_{yT}}\lesssim T\|u\|_{L^\infty_TH^1},
	\end{equation}
	\begin{equation}\label{l9}
	\|E_5(u)\|_{L^2_xL^\infty_{yT}}\lesssim c(1,T)T\|u\|_{L^\infty_TH^1},
	\end{equation}
	and
	\begin{equation}\label{l10}
	\|\partial_x\mathcal{L}(E_5(u))\|_{L^2_xL^\infty_{yT}}+\|\partial_y\mathcal{L}(E_5(u))\|_{L^2_xL^\infty_{yT}}\lesssim c(1,T)T\|u\|_{L^\infty_TH^1}.
	\end{equation}
\end{lemma}
\begin{proof}
	We only estimate \eqref{l6}. The other estimates are similar to those in Lemma \ref{lemmaE4}. Since $\partial_x\mathcal{L}$ is a bounded operator on $L^2(\mathbb{R}^2)$, we have
	\begin{equation*}
		\begin{split}
			\|E_5(u)(t)\|_{L^2_{xy}}&\leq \int_0^T\|\varphi_x\partial_x\mathcal{L}(u)\|_{L^2_{xy}}dt'\\
			&\leq \int_0^T\|\varphi_x\|_{L^\infty_x}\|\partial_x\mathcal{L}(u)\|_{L^2_{xy}}dt'\\
			&\leq \int_0^T\|\varphi_x\|_{L^\infty_x}\|u\|_{L^2_{xy}}dt'\lesssim T\|u\|_{L^\infty_TH^1}.
		\end{split}
	\end{equation*}
	Now, since $\partial_x$ and $\partial_y$ commutes with $\partial_x\mathcal{L}$ and using that $\partial_x\mathcal{L}$ is a bounded operator on $L^2(\mathbb{R}^2)$ once again, we deduce
	\begin{equation*}
		\begin{split}
			\|\partial_xE_5(u)(t)\|_{L^2_{xy}}&\leq \int_0^T\|\partial_x(\varphi_x\partial_x\mathcal{L}(u))\|_{L^2_{xy}}dt'\\
			&\leq \int_0^T\|\varphi_{xx}\partial_x\mathcal{L}(u)+\varphi_{x}\partial_x\mathcal{L}(\partial_xu)\|_{L^2_{xy}}dt'\\
			&\leq T(\|\varphi_{xx}\|_{L^\infty_{xT}}\|u\|_{L^\infty_TL^2_{xy}}+\|\varphi_{x}\|_{L^\infty_{xT}}\|\partial_xu\|_{L^\infty_T L^2_{xy}})\\
			&\lesssim T\|u\|_{L^\infty_TH^1}.
		\end{split}
	\end{equation*}
	and
	\begin{equation*}
		\begin{split}
			\|\partial_yE_5(u)(t)\|_{L^2_{xy}}&\leq \int_0^T\|\varphi_x\partial_x\mathcal{L}(\partial)yu)\|_{L^2_{xy}}dt'\\
			&\leq T\|\varphi_{x}\|_{L^\infty_{xT}}\|\partial_yu\|_{L^\infty_T L^2_{xy}}\\
			&\lesssim T\|u\|_{L^\infty_TH^1}.
		\end{split}
	\end{equation*}
	Combining the above estimates we obtain \eqref{l6}. The proof of the lemma is thus completed.
\end{proof}

With the above estimates in hand, we are able to prove Theorem \ref{lwp1}.

\begin{proof}[Proof of Theorem \ref{lwp1}]
As we  already mentioned, we use the fixed point theorem. Let  us define the  operator
\begin{equation*}
\Phi(u)(t)=U(t)u_0-\dfrac{1}{2}\sum_{j=1}^5E_j(u)(t),
\end{equation*}
and          the closed   ball
\begin{equation*} 
X_T^a:=\{ v\in C([0,T]:H^1(\R^2)):\tres v \tres \le a\},
\end{equation*}
where
\begin{equation*}
\begin{split}
\tres  v \tres&:=\|v\|_{L^{\infty}_TH^1}+\|\p^2v\|_{\kse}
+\|\nabla v\|_{\str}\\
&\quad+\|v\|_{\mne}+\|\p_x \ol(v)\|_{\mne}+\|\p_y\ol(v)\|_{\mne}.
\end{split}
\end{equation*}

By using Lemmas \ref{l0}--\ref{lemmaE5}, we are able to show that there  exist positive constants $a$ and $T$ such that $\Phi:X_T^a\mapsto X_T^a$ is well defined and is a contraction. From this point on, the arguments to  complete  the  proof of the theorem are standard. So we will not give the details.
\end{proof}

\begin{remark}\label{remwelp}
By using the same strategy as in the proof of Theorem \ref{lwp1} we may show that \eqref{uzkp} is locally well-posed in the anisotropic Sobolev space $H^{s,1}(\R^2)$, $s>3/4$, which is defined through the norm
$$
\|v\|_{H^{s,1}}^2=\|v\|^2_{L^2_{xy}}+\|D_x^sv\|^2_{L^2_{xy}}+\|\p_yv\|^2_{L^2_{xy}}.
$$
Here $D_x^{s}$,  is defined through its Fourier transform as $\F_{xy}({D_x^{s} u_0})(\xi,\mu)=|\xi|^{s}\F_{xy}(u_0)(\xi,\mu)$.
	
To see this, it suffices to show that the integral equation \eqref{ie} has a unique fixed point in the ball
$$
X_T^a:=\{ v\in C([0,T]:H^{s,1}(\R^2)):\tres v \tres \le a\}
$$
where now
\[
\begin{split}
\tres  v \tres&:=\|v\|_{L^{\infty}_TH^{s,1}}+\|\nabla D^s_x v\|_{\kse}+\|\nabla \p_y v\|_{\kse}
+\|\nabla v\|_{L^2_TL^\infty_{xy}}\\
&\quad+\|v\|_{\mne}+\|\p_x \ol(v)\|_{\mne}+\|\p_y\ol(v)\|_{\mne}.
\end{split}
\]
The nonlinear estimates are similar to those ones for the terms $E_1$, $E_2$ and $E_3$ above,  except the ones including the operator $D^s_x$, where we need to use the  one-dimensional fractional Leibniz rule (see \cite{kpv1}):
\begin{equation}\label{leibrule}
\|D_x^s(fg)-fD_x^sg-gD_x^sf\|_{L^2_x}\leq c \|g\|_{L^\infty_x}\|D^sf\|_{L^2_x}, \qquad s\in(0,1),
\end{equation}
and the estimate 
\begin{equation}\label{str1}
\|U(t)u_0\|_{L^2_TL^\infty_{xy}}\leq c \|D_x^{-\varepsilon/2} u_0\|_{L^2_{xy}},
\end{equation}
which holds for any $0\leq \varepsilon<1/2$.  Estimate \eqref{str1} was established in \cite{LP}.

For the sake of clearness we estimate $E_1$ in the norms $\|D_x^s(E_1(u))\|_{L^2_{xy}}$ and $\|\p_xE_1(u)\|_{L^2_TL^\infty_{xy}}$, assuming $s\in(3/4,1)$. Indeed, from group properties, Minkowski and H\"older's inequalities, we obtain
\begin{equation}\label{E1lei}
\|D_x^s(E_1(u))\|_{L^2_{xy}}\lesssim \int_0^T\|D_x^s(uu_x)\|_{L^2_{xy}}dt\lesssim T^{1/2}\|D_x^s(uu_x)\|_{L^2_{xyT}}dt.
\end{equation}
Now, we use \eqref{leibrule} to obtain
\[
\begin{split}
\|D_x^s(uu_x)\|_{L^2_{xyT}}&\leq \| \|D_x^s(uu_x)-uD_x^su_x-u_xD_x^su\|_{L^2_x} \|_{L^2_{yT}}+\|uD_x^su_x\|_{L^2_{xyT}}+\|u_xD_x^su)\|_{L^2_{xyT}}\\
&\lesssim  \|u_x\|_{L^2_TL^\infty_{xy}}\|D^s_xu\|_{L^{\infty}_TL^{2}_{xy}}+c\|u\|_{\mne}\| D^s_x u_x\|_{\kse}.
\end{split}
\]
The last  two  inequalities combine to give
$$
\|D_x^s(E_1(u))\|_{L^2_{xy}}\lesssim T^{1/2}\tres u\tres^2.
$$
Also, in view of \eqref{str1}, we deduce
\begin{equation}\label{recder}
\|\p_x(E_1(u))\|_{L^2_TL^\infty_{xy}}\lesssim  \int_0^T\|D_x^{-\varepsilon/2}\p_x(uu_x))\|_{L^2_{xy}}dt.
\end{equation}
Since $s>3/4$, we can choose $\varepsilon$ sufficiently close to $1/2$ in such a way that $1-\varepsilon/2\leq s$. As a consequence,
$$
\|\p_x(E_1(u))\|_{L^2_TL^\infty_{xy}}\lesssim  \int_0^T\|uu_x\|_{H^{s,1}}dt.
$$
From this point on, one uses the same arguments as in \eqref{E1lei} and in the proof of Theorem \ref{lwp1}.
\end{remark}

\section*{Acknowledgement}

FL was partially supported by FAPERJ-Brazil and CNPq-Brazil. AP is partially supported by FAPESP-Brazil and CNPq-Brazil.

\end{document}